\documentclass[10pt]{amsart}
\usepackage[margin=1.2in]{geometry}
\usepackage{amsmath,amsthm,amssymb}
\usepackage[mathscr]{eucal}
 \usepackage{cite}
\usepackage{upgreek}
\usepackage[bookmarks=false]{hyperref}
\usepackage{enumerate}
\usepackage{mathtools}

\numberwithin{equation}{section}

\newtheorem{thm}{Theorem}[section]
\newtheorem*{thm*}{Theorem}
\newtheorem{lem}[thm]{Lemma}
\newtheorem*{prob*}{Problem}

\newtheorem{prop}[thm]{Proposition}
\newtheorem*{prop*}{Proposition}
\newtheorem{conj}[thm]{Conjecture}
\newtheorem{cor}[thm]{Corollary}
\newtheorem*{cor*}{Corollary}

\theoremstyle{definition}
\newtheorem{defn}[thm]{Definition}
\newtheorem*{defn*}{Definition}

\newtheorem{remark}[thm]{Remark}

\newtheorem*{question*}{Question}
\newtheorem*{Pquestion*}{Popa's question}

\newtheorem*{conv*}{Convention}

\newcommand{\N}{\mathbb{N}}
\newcommand{\F}{\mathbb{F}}
\newcommand{\R}{\mathbb{R}}
\newcommand{\C}{\mathbb{C}}

\newcommand{\M}{\mathbb{M}}
\newcommand{\E}{\mathbb{E}}

\renewcommand{\P}{\mathbb{P}}

\newcommand{\cH}{\mathcal{H}}

\newcommand{\cK}{\mathcal{K}}
\newcommand{\cM}{\mathcal{M}}

\newcommand{\cU}{\mathcal{U}}

\newcommand{\op}{\operatorname{op}}
\newcommand{\tr}{\operatorname{tr}}
\newcommand{\Tr}{\operatorname{Tr}}

\newcommand{\Ad}{\operatorname{Ad}}

\newcommand{\Prob}{\operatorname{Prob}}

\DeclareMathOperator{\re}{Re}

\DeclarePairedDelimiter{\ip}{\langle}{\rangle}
\newcommand{\norm}[1]{\left\Vert #1\right\Vert}

\begin{document}

\title[Questions on the structure of random embeddings of $L(\mathbb{F}_2)$]
{Questions on the structure of random embeddings of $L(\mathbb{F}_2)$}

\author[B. Hayes]{Ben Hayes}
\address{Department of Mathematics, University of Virginia\\
141 Cabell Drive, Kerchof Hall
P.O. Box 400137,
Charlottesville, VA 22904}
\email{brh5c@virginia.edu}
\urladdr{https://sites.google.com/site/benhayeshomepage/}

\author[D. Jekel]{David Jekel}
\address{Department of Mathematical Sciences, University of Copenhagen, Universitetsparken 5, 2100 Copenhagen {\O}, Denmark}
\email{daj@math.ku.dk}
\urladdr{https://davidjekel.com}

\author[S. Kunnawalkam Elayavalli]{Srivatsav Kunnawalkam Elayavalli}
\address{Department of Mathematics, UC Berkeley, 970 Evans Hall, Berkeley, CA 94720, USA}\email{sriva@umd.edu}
\urladdr{https://sites.google.com/view/srivatsavke/home}

\begin{abstract}
Motivated by recent developments at the interface of operator algebras and random matrix theory, we propose new conjectures concerning the asymptotic structure of random matrix models of the countable free groups. The first conjecture predicts a random matrix analogue of the Akemann-Ostrand property for free groups, and reveals a succinct approach to recover the Peterson-Thom property for $L(\mathbb{F}_2)$.  The second stronger conjecture is motivated by continuous model theory. It predicts that the \emph{random} embedding of the free group factor into a matrix ultraproduct is \emph{existential}. We discuss the interesting relationship between these conjectures.  
\end{abstract}

\maketitle


\section{Introduction}
Given $\alpha=\sum_{w_{1},w_{2}\in \F_{r}}\alpha_{w_{1},w_{2}}w_{1}\otimes w_{2}\in \C[\F_{r}]\otimes \C[\F_{r}]$ (here and throughout $\F_{r}$ is the free group on $r$ letters), and  two tuples $u=(u_{j})_{j=1}^{r},(v_{j})_{j=1}^{r}$ of unitary operators such that $u_{i}v_{k}=v_{k}u_{i}$ for all $1\leq i,k\leq r$, we use 
\[\alpha(u,v)=\sum \alpha_{w_{1},w_{2}}w_{1}(u)w_{2}(v).\]
A natural problem from the perspective of random matrices is to understand the almost sure limits of the eigenvalue distributions of $\alpha(U^{(N)}\otimes 1_N, 1_N\otimes (U^{(N)})^{t})$ for, say, self-adjoint elements $\alpha\in C[\F_r]\otimes\C[\F_r]$, where $U^{(N)}\in \cU(\M_{N}(\C))^{r}$ is a random $N\times N$ unitary distributed according to Haar measure on $\cU(\M_{N}(\C))^{r}$. Let $a_{1},\cdots,a_{r}$ be the generators of $\F_{r}$, and $\lambda\colon \F_{r}\to\cU(B(\ell^{2}(\F_{r}))$. Then by Voiculescu's asymptotic freeness theorem \cite[Theorem 3.8]{VoicAsyFree}, we know that the noncommutative distribution of  $U^{(N)}$ converges weak$^{*}$ to the distribution of $\lambda(a):=(\lambda(a_{j}))_{j=1}^{r}$ with respect to the trace given by $\ip{\cdot \delta_{e},\delta_{e}}$. Using the notation $\lambda(a^{-1})=(\lambda(a_{j}^{-1}))_{j=1}^{r}$, this weak$^{*}$-convergence immediately implies that the eigenvalue distribution of $\alpha(U^{(N)}\otimes 1_N, 1_N\otimes (U^{(N)})^{t})$ converges weak$^{*}$ to the spectral measure of $\alpha(\lambda(a)\otimes 1,1\otimes \lambda(a^{-1}))$ acting $\ell^{2}(\F_{r})\otimes \ell^{2}(\F_{r})$ with respect to the vector $\delta_{e}\otimes \delta_{e}$. Thus the key difficulty here is to understand outlier eigenvalues in the eigenvalue distribution. This puts one immediately into the realm of the strong convergence phenomena, as first developed by Haagerup-Thorbj\o rnsen \cite{HTExt}, namely the study of asymptotics of operator norms of noncommutative $*$-polynomials in random matrices. This theory has seen a lot of activity in recent years, a sample set of which includes \cite{ BordenaveCollins2019, bordenave2023norm, collins2014strong, chen2024newapproachstrongconvergence, Chen_2026, gao2026newsource,loudermagee2025limitgroups, MageeThomas2023, MdLSstrongasymptoticfreenesshaar}. We direct the reader to the surveys \cite{BCICM, magee2025strong, vHCDM, vHICM, bordenave2025sparsegraphsbenjaminischrammlimits} for more information concerning this program. 

Na{\"\i}vela, one would still expect that the distribution $(U^{(N)}\otimes 1_N, 1_N\otimes (U^{(N)})^{t})$ converges strongly to $(\lambda(a)\otimes 1,1\otimes \lambda(a^{-1}))$. However, this is quickly seen to be false. Indeed, consider $\alpha=\sum_{j=1}^{r}a_{j}\otimes a_{j}^{-1}$, where $a_{1},\cdots,a_{r}$ are the generators of $\F_{r}$. for every integer $N$, we have that 
\[\left\|\frac{1}{2r}\sum_{j=1}^{r}U_{j}^{(N)}\otimes \overline{U_{j}^{(N)}}+\frac{1}{2r}\sum_{j=1}^{r}(U_{j}^{(N)})^{-1}\otimes (U_{j}^{(N)})^{t}\right\|=1,\]
(e.g. by acting on the identity matrix), whereas by \cite[Theorem 3]{KestenSR},
\[\left\|\frac{1}{2r}\sum_{j=1}^{r}
\lambda(a_{j})\otimes \lambda(a_{j})+\sum_{j=1}^{r}
\lambda(a_{j})^{-1}\otimes \lambda(a_{j})^{-1}\right\|=\frac{\sqrt{2r-1}}{2r}<1.\]
Consideration of the asymptotics of the operator norm after applying this specific $\alpha$ to random matrices {and then restricting to the orthogonal complement of the identity matrix} leads to the notion of \emph{quantum expanders}, and Hasting's result \cite{Hastings} (see also \cite{PisierQuantumExpander}) asserts that, almost surely, we have 
\[\lim_{N\to\infty}\left\|\left(\frac{1}{2r}\sum_{j=1}^{r}U_{j}^{(N)}\otimes \overline{U_{j}^{(N)}}+\sum_{j=1}^{r}(U_{j}^{(N)})^{-1}\otimes (U_{j}^{(N)})^{t}\right)\big|_{(\C 1)^{\perp}}\right\|=\frac{\sqrt{2r-1}}{2r}.\]
More generally, the work of Bordenave-Collins \cite{BCPerm, bordenave2023norm} asserts that if we restrict the tuple $(U_{j}^{(N)}\otimes \overline{U_{j}^{(N)}})_{j=1}^{r}$  to $(\C 1)^{\perp},$ then it converges strongly to an $r$-tuple of freely independent Haar unitaries. This still leaves open the case of more general polynomials in $(U^{(N)}\otimes 1_{N},1_{N}\otimes (U^{(N)})^{t})$. Here, the main obstruction to strong convergence is the observation that if $\beta\in \C(F_{r})$ is given, then
\[\lim_{N\to\infty}\left\|\alpha(U^{(N)}\otimes 1,1\otimes (U^{(N)})^{t})\#\beta(U^{(N)})\right\|_{2}=\]\[\left\|\alpha((\lambda(a_{j}))_{j=1}^{n},(\rho(a_{j}^{-1}))_{j=1}^{n})\left(\sum_{g}\beta_{g}\delta_{g}\right)\right\|_{2},\]
where $\beta=\sum_{g}\beta_{g}g$, and $\lambda,\rho$ are the left/right regular representations of $\F_{r}$.  Thus,
\[
\liminf_{N\to\infty}\left\|\alpha(U^{(N)}\otimes 1,1\otimes (U^{(N)})^{t}) \right\| \geq \left\|\alpha((\lambda(a_{j}))_{j=1}^{n},(\rho(a_{j}^{-1}))_{j=1}^{n}) \right\|_{B(\ell^2(\F_r))}.
\]
We always have that 
\[\|\alpha(\lambda(a)\otimes 1,1\otimes \lambda(a^{-1}))\|_{B(\ell^{2}(\F_{r})\otimes \ell^{2}(\F_{r}))}\leq \|\alpha((\lambda(a_{j}))_{j=1}^{n},(\rho(a_{j}^{-1}))\|_{B(\ell^{2}(\F_{r}))}\,\]
and equality does not always hold (e.g. the two sides are not equal for $\alpha=\sum_{j}a_{j}\otimes a_{j}^{-1}$). Thus if we want strong convergence of $(U^{(N)}\otimes 1_{N},1_{N}\otimes (U^{(N)})^{t})$ to $(\lambda(a)\otimes 1,1\otimes \lambda(a))$, then we need to remove not just scalar multiples of the identity, but also $*$-polynomials in the $U^{(N)}$ themselves.
 We state a conjecture which asserts precisely that this is the only obstruction, and once such polynomials are removed, then we have strong convergence to $(\lambda(a)\otimes 1,1\otimes \lambda(a))$.

\begin{conj}\label{intro conj: the conj}
	\label{scmlskmc}
	Given $\alpha\in \C[\F_r]\otimes\C[\F_r]$, and $u$ an $r$-tuple of free Haar unitaries for almost every $U^{(N)}\in \prod_{N}\cU(\cM_{N}(\C)^{r})$ the following estimate holds:
	$$ \sup_{R>0} \inf_{F\subset \C[\F_r] \text{ finite}}\ \limsup_{N\to\infty} \sup_{\substack{A\in\M_N(\C), \norm{A}\leq R, \norm{A}_2\leq 1,\\ \forall P\in F, \tr(A^*P(U^N))=0}} \norm{\alpha(U^{(N)}\otimes 1_N, 1_N\otimes (U^{(N)})^{t}) \# A }_2 $$
    $$\leq \norm{\alpha(u\otimes 1, 1\otimes u^{-1})}_{C^{*}(u)\otimes_{\min} C^{*}(u) },$$
	where $(B\otimes C)\#A = BAC^T$.
\end{conj}

Our motivation for the above conjecture comes from the realm of operator algebras, particularly the Akemann-Ostrand property \cite{AOProp}, and the Peterson-Thom conjecture \cite{PetersonThom}. As we show in Corollary \ref{cor: PT corollary}, Conjecture \ref{intro conj: the conj} would give a different proof of the Peterson-Thom property for free group factors, asserting that any diffuse amenable subalgebra of a free group factor has a unique maximal amenable extensions. The new approach stated here distills the operator algebraic side of the argument as much as possible, while efficiently exporting many of the difficulties to the random matrix side. Recall that the first-named author in \cite[Theorem 1.1]{HayesPT} gave a random matrix conjecture about the strong convergence of $(U^{(N)}\otimes 1_{N},1_{N}\otimes V^{(N)})$ where $U^{(N)}$ and $V^{(N)}$ are \emph{independent} $r$-tuples of unitaries each distributed according to the Haar measure on $\cU(\M_{N}(\C))^{r}$ which he showed implied the Peterson-Thom conjecture. This random matrix conjecture of the first-named author has since been resolved by several sets of authors \cite{PTkilled, bordenave2023norm, MdLSstrongasymptoticfreenesshaar, chen2024newapproachstrongconvergence,  Parraud2024strong} and the methods involved have also had applications in geometry, random graph theory, and random representation theory.
We expect that our alternative resolution via  Conjecture \ref{intro conj: the conj} will also have other applications, for instance to the structure of free products of II$_1$ factors. We remark that there is additional motivation from random matrices. It is well known that $(U_{i}^{(N)}),(V_{i}^{(N)})$ are asymptotically free, where $(V_{i}^{(N)})$ are an independent copy of $(U_{i}^{(N)})_{i}$, and we now know by \cite{PTkilled, bordenave2023norm, chen2024newapproachstrongconvergence, MdLSstrongasymptoticfreenesshaar, Parraud2024strong} that $(U_{i}^{(N)}\otimes 1,1\otimes V_{i}^{(N)})$ strongly converge to $(\lambda(a)\otimes, 1\otimes \lambda(a_{i}))$. 
By \cite{MPFreeTranspose} we have that $(U_{i}^{(N)}),((U_{i}^{(N)})^{t})_{i=1}^{r}$ are asymptotically free, and from this it is plausible that $(U_{i}^{(N)}\otimes 1_{N},1_{N}\otimes (U_{i}^{(N)})^{t})_{i=1}^{r}$ strongly converge to $(\lambda(a)\otimes, 1\otimes \lambda(a))$, after we remove the subspace of polynomials in $(U_{i}^{(N)})_{i=1}^{r}$. 

The connection between Conjecture \ref{intro conj: the conj} and the Peterson-Thom conjecture goes through ultraproducts of matrix algebras. Given a sequence $(Q_{n},\tau_{n})$ and a free ultrafilter $\omega\in \beta\N\setminus\N$ of tracial von Neumann algebras, we define their ultrapower by
\[\prod_{n\to\omega}(Q_{n},\tau_{n}):=\frac{\{(a_{n})_{n}\in \prod_{n}N_{n}:\sup_{n}\|a_{n}\|<+\infty\}}{\{(a_{n})_{n}\in \prod_{n}N_{n}:\sup_{n}\|a_{n}\|<+\infty,\lim_{n\to\omega}\tau_{n}(a_{n}^{*}a_{n})=0\}}.\]
By the same argument as in \cite[Lemma A.9]{BO}, the ultraproduct is a von Neumann algebra with a faithful, normal tracial state
\[\tau_{\omega}((a_{n})_{n\to\omega})=\lim_{n\to\omega}\tau_{n}(a_{n}),\]
where we use $(a_{n})_{n\to\omega}$ for the image in $\prod_{n\to\omega}Q_{n}$ of a norm bounded sequence in $\prod_{n}Q_{n}$.
The main point of the ultraproduct is to convert properties holding asymptotically into properties holding exactly in the ultraproduct. E.g. Voiculescu's asymptotic theorem implies that for almost every $[U^{(N)}]_{N\in \N}\in \prod_{N}\cU(\M_{N}(\C))^{r},$ and for every free ultrafilter $\omega\in\beta\N\setminus\N$, there is a unique embedding $\Theta_{U}\colon L(\F_{r})\to \prod_{n\to\omega}(\M_{n}(\C),\tr)$ so that $\Theta_{U}(\lambda(a_{j}))=(U_{j}^{(N)})_{N\to\omega}$ for all $j=1,\cdots,r$. 

An embedding $Q\leq M$ of tracial von Neumann algebras is \emph{weakly coarse} if the orthocomplement $Q$-$Q$ bimodule $L^{2}(M)\ominus L^{2}(Q)$ is weakly contained in an infinite direct sum of the coarse bimodule, $L^2(Q)\otimes L^2(Q)$. E.g. if $Q$ has no amenable direct summand, this asserts a weak version of spectral gap for $Q\leq M$ (see \cite{Popa_weakspgap}), and as such can be viewed of a strengthening of ``strong ergodicity" of the inclusion $Q\leq M$, similarly to how an irreducible inclusion (i.e. one where $Q'\cap M=\C1$) is analogous to an ergodic action of a group.

As we show later, Conjecture \ref{intro conj: the conj} is equivalent to saying that, almost surely, the random embedding $\Theta_{U,\omega}$ is weakly coarse. Interestingly, this can be viewed as a random version of the influential Akemann-Ostrand property. Indeed, by a recent result of Ding-Peterson \cite[Theorem 7.20]{DP22} (see also \cite[Corollary 3.7]{marrakchi2025kadison}), it turns out that the Akemann-Ostrand property for $L(\mathbb{F}_2)$ is equivalent to the diagonal embedding of $L(\mathbb{F}_2)\subset L(\mathbb{F}_2)^{\omega}$ being weakly coarse. Conceptually, this analogy is quite satisfying because Ozawa showed that the AO property implies solidity \cite{OzawaSolid}, while our random AO property would imply the Peterson-Thom property which is can be viewed as an optimal strengthening of solidity.

Motivated by the above connection to the AO property, we are able to formulate the following stronger conjecture about random unitaries, involving not only traces of polynomials in $r$-variables, but traces of self-adjoint polynomials in $r+s$ variables, where we optimize over the $s$ auxiliary variables before substituting the random unitaries in the first $r$ variables.

\begin{conj} \label{intro conj: random optimization conjecture}
Let $U_1^{(N)}$, \dots, $U_r^{(N)}$ be independent Haar random unitaries.  Then for every $s \in \N$ and every self-adjoint $*$-polynomial in $r + s$ variables, we have
\begin{multline} \label{eq: random optimization conjecture}
\lim_{N \to \infty} \mathbb{E} \left[ \sup_{V_1,\dots,V_s \in \cU(\M_{N}(\C))} \tr_N(p(U_1^{(N)},\dots,U_r^{(N)},V_1,\dots,V_s)) \right] \\ = \sup_{v_1,\dots,v_s \in \mathcal{U}(L(\mathbb{F}_r))} \tau(p(u_1,\dots,u_r,v_1,\dots,v_s)).
\end{multline}
\end{conj}

In Theorem \ref{thm:intro main thm} (2), we will relate Conjecture \ref{intro conj: random optimization conjecture} to the concept of existential embeddings of operator algebras, which comes from continuous model theory \cite{FHS2014a}. An embedding $Q\leq M$ of tracial von Neumann algebras is \emph{existential} if a trace-preserving embedding $\theta\colon M\to Q^{\omega}$ so that $\theta|_{Q}$ is the diagonal embedding; here if $Q$ and $M$ are separable, then $\omega$ can be any ultrafilter on the natural numbers, but in the non-separable case, one must use an ultrafilter on a larger index set.  The relevant case for this paper is when $Q$ is separable and $M$ is not separable (e.g. an ultraproduct), so we use the following characterization:  For $Q$ with separable predual and a given ultrafilter $\omega$ on $\mathbb{N}$, the embedding $Q \leq M$ being existential means that for every separable $M_0$ with $Q \leq M_0 \leq M$, there exists an embedding $M_0 \to Q^\omega$ extending the diagonal embedding of $Q$.

Existential embeddings have many applications in operator algebras.  For instance, the McDuff property is equivalent to the embedding of $M$ into $M_2(\mathbb{C}) \otimes M$ being existential.  Popa's theorem on asymptotic free independence \cite{popa1995free} implies that for every $\mathrm{II}_1$ factor $M$, the embedding of $M$ into $M * M$ is existential.  The analogous concept of existential embeddings into free products for the $\mathrm{C}^*$-algebraic setting was formulated in Robert's notion of selflessness \cite{robertselfless}, which led to the resolution of the resolution to Blackadar's strict comparison problem for $C^*_r(\mathbb{F}_2)$ \cite{AGKEP}, and is a current area of active research.

Conjecture \ref{intro conj: random optimization conjecture} states in a precise sense that the ``random embedding is existential'', as we explain below in Theorem \ref{thm:intro main thm}. Note that even the existence of an existential embedding of $L(\mathbb{F}_2)$ into the matrix ultraproduct is open (see \cite{elayavalli2025}).  On the other hand, in the reverse direction, a recent result of Peterson shows that it is not possible for some $M$ elementarily equivalent to $\prod_{N \to \omega} \mathbb{M}_N$ to existentially embed into $L(\mathbb{F}_r)^\omega$ \cite[Theorem 7.10]{Peterson2026}, which indicates the difficulty of Conjecture \ref{intro conj: random optimization conjecture}.  Meanwhile, \cite{jekel2024upgradedfreeindependencephenomena} gives examples of nontrivial sup-formulas for $u_1$, \dots, $u_r$ which have the same qualitative behavior in $L(\mathbb{F}_r)$ and in the matrix ultraproduct, namely, sup-formulas which express that for all $v_1$, \dots, $v_r$, if $[u_j,v_j] = 0$, then $v_1$, \dots, $v_m$ are freely independent; see \cite[Remark 4.5]{jekel2024upgradedfreeindependencephenomena}.  This could be taken as positive evidence of Conjecture \ref{intro conj: random optimization conjecture}.  We also point out a connection with free entropy theory in the work of the second author in \cite[Theorem E, Corollary F]{Jekel2026stochastic} (see also \cite[\S 6]{GJNP2025}), which relates free entropy with certain optimization formulas generalizing \eqref{eq: random optimization conjecture}.  The optimization problems there are stochastic in nature; the role of $U_j^{(N)}$ is played by a self-adjoint matrix Brownian motion, and the role of the control parameters $V_j$ is played by a control process adapted to a given noncommutative filtration.  An analog of Conjecture \ref{conj: random optimization conjecture} in this setting would imply that the value of matrix optimization problems converges as $N \to \infty$, and would yield as a consequence a complete large deviations principle for several independent GUE matrices by \cite[Corollary F]{Jekel2026stochastic}.

We now state our main result, which precisely relates the conjectures stated above with their operator algebraic counterparts, and with the Peterson--Thom conjecture. 

\begin{thm} \label{thm:intro main thm}
Fix an integer $r$.
\begin{enumerate}
    \item Conjecture \ref{intro conj: the conj} is equivalent to the statement there is a conull $\Omega\subseteq \prod_{N}\cU(\M_{N}(\C))^{r}$ so that for every $U=[U^{(N)}]_{N\in \N}\in \Omega$ so that for every $\omega\in \beta\N\setminus\N$ there is a weakly coarse trace-preserving embedding $\Theta_{U,\omega}\colon L(\F_{r})\to \prod_{N\to\omega}(\M_{N}(\C),\tr)$ satisfying that $\Theta_{U,\omega}(u_{j})=(U_{j,N})_{N\to\omega}$ for all $j=1,\cdots, r$. \label{item: intro item weak coarse equiv}
    \item Conjecture \ref{intro conj: random optimization conjecture} is equivalent to the statement that there is a conull equivalent to the statement there is a conull $\Omega\subseteq \prod_{N}\cU(\M_{N}(\C))^{r}$ so that for every $U=[U^{(N)}]_{N\in \N}\in \Omega$ and for every $\omega\in \beta\N\setminus\N$ there is an existential trace-preserving embedding $\Theta_{U,\omega}\colon L(\F_{r})\to \prod_{N\to\omega}(\M_{N}(\C),\tr)$.
    \item Conjecture \ref{intro conj: random optimization conjecture} implies Conjecture \ref{intro conj: the conj}. \label{item: intro conjecture bridge}
    \item Conjecture \ref{intro conj: the conj} implies the Peterson-Thom property for $L(\F_{r})$.  \label{item: intro item PT}
\end{enumerate}
\end{thm}

The main ingredient in the proof that Conjecture \ref{intro conj: random optimization conjecture} implies Conjecture \ref{intro conj: the conj} is Ding and Peterson's result, which says that the embedding of $L(\F_r)$ into $L(\F_r)^\omega$ is coarse.  If the embedding of $L(\F_r)$ into $\prod_{n \to \omega} \mathbb{M}_N$ is existential, then the coarseness passes to the inclusion $L(\F_r) \subseteq \prod_{n \to \omega} \mathbb{M}_N$, which yields Conjecture \ref{intro conj: the conj}.

\subsection*{Acknowledgements}

The authors thank Charles Bordenave, Beno{\^\i}t Collins,  Jorge Garza-Vargas, F{\'e}lix Parraud, and Ramon van Handel for insightful discussions. Conversations at the IPAM conference ``Free Entropy Theory and Random Matrices" were inspirational for this work.

\subsection*{Funding}

DJ was supported by an EU Horizon Marie Sk{\l}odowska Curie Action, FREEINFOGEOM, grant id: 101209517.  Views expressed in this paper are those of the authors, not of the funding agencies. SKE was supported by the NSF grant DMS-2611847. BH acknowledges support from NSF CAREER award DMS-21447.


\subsection*{AI statement} 

No AI tools were used by the authors in any stage in the process of carrying out this research project. 


\section{Preliminaries} \label{sec:preliminaries}

\subsection*{Ultraproducts of von Neumann algebras, and bimodules}

By definition, a \emph{tracial von Neumann algebra} is a pair $(M,\tau)$ where $M$ is a von Neumann algebra and $\tau\colon M\to \C$ is a faithful, normal, tracial state.

Ultraproducts of von Neumann algebras will be useful throughout the paper.
\begin{defn}
Let $\omega$ be a free ultrafilter on $\N,$ and let $(M_{k},\tau_{k})_{k=1}^{\infty}$ be a sequence of tracial von Neumann algebras. We define their \emph{tracial ultraproduct with respect to $\omega$} by
\[\prod_{k\to\omega}(M_{k},\tau_{k})=\frac{\{(x_{k})_{k}\in \prod_{k}M_{k}:\sup_{k}\|x_{k}\|_{\infty}<\infty\}}{\{(x_{k})_{k}\in \prod_{k}M_{k}:\sup_{k}\|x_{k}\|_{\infty}<\infty,\mbox{ and}\lim_{k\to\omega}\|x_{k}\|_{L^{2}(\tau_{k})}=0\}}.\]
If $(x_{k})_{k}\in \prod_{k}M_{k}$ and $\sup_{k}\|x_{k}\|_{\infty}<\infty,$ we let $(x_{k})_{k\to\omega}$ be the image of $(x_{k})_{k}$ under the quotient map.
If $J$ is an index set, and $(x_{k})_{k}\in \prod_{k}M_{k}^{J}$ and
\[\sup_{k}\|x_{k,j}\|_{\infty}<\infty\mbox{ for all $j\in J$},\]
then we let $(x_{k})_{k\to\omega}\in \left(\prod_{k\to\omega}(M_{k},\tau_{k})\right)^{J}$ be the tuple whose j$^{th}$ coordinate is $(x_{k,j})_{k\to\omega}.$
\end{defn}

As is well known,  $\prod_{k\to\omega}(M_{k},\tau_{k})$ is a tracial von Neumann algebra with the $*$-algebra operations defined pointwise and the trace given by $\tau_{\omega}((x_{k})_{k\to\omega})=\lim_{k\to\omega}\tau_{k}(x_{k})$.

We recall the following definitions about bimodules over von Neumann algebras (see for instance \cite[Appendix F]{BrownOzawa2008}).  For von Neumann algebras $M$ and $N$, an \emph{$M$-$N$-bimodule} is a Hilbert space $H$ equipped with a normal left action of $M$ and right action of $N$. For a von Neumann algebra $M$ with faithful normal state $\varphi$, the \emph{trivial bimodule} $L^2(M,\varphi)$ (often denoted simply $L^2(M)$) is the GNS space equipped with the left and right actions $x \cdot \xi \cdot y = x JyJ \xi$ where $J$ is the modular conjugation operator. \emph{The coarse $M$-$N$ bimodule} is the bimodule $L^2(M) \otimes_{\C} L^2(N)$.  More generally, an $M$-$N$-bimodule $H$ is said to be \emph{coarse} if $M$ embeds into a direct sum of copies of $L^2(M) \otimes_{\C} L^2(N)$. Let $\cH,\cK$ be $M$-$N$ bimodules and let $\pi_{\cH}\colon M\otimes_{\textnormal{alg}}N^{\op}\to B(\cH)$ be given by $\pi_{\cH}(x\otimes y^{\op})\xi=x\xi y$, and similarly define $\pi_{\cK}$. We say that $\cH$ is \emph{weakly contained in} $\cK$ if 
\[
\|\pi_{\cH}(a)\|\leq \|\pi_{\cK}(a)\|, \textnormal{ for all $a\in M\otimes_{\textnormal{alg}}N^{\op}$}
\]
 (see \cite[Section~13.3.2]{anantharaman-popa}).  
%
In particular, an $M$-$N$ bimodule is \emph{weakly coarse} if it is weakly contained in the coarse bimodule.  Note that weak containment is transitive, so in particular any bimodule that is weakly contained in a weakly coarse bimodule is also weakly coarse.

\subsection*{Haar random unitaries, concentration, and asymptotic freeness}

For $N \in \N$, we denote by $\mathbb{M}_N(\mathbb{C})$ the space of $N \times N$ complex matrices.  We denote by $\tr$ (or $\tr_N$) the normalized trace $(1/N) \Tr$.  In particular $(\mathbb{M}_N(\mathbb{C}),\tr)$ is a tracial von Neumann algebra.

For $N,r\in \N$, note that $\cU(\M_{N}(\C))^{r}$ is a compact group, and thus has a unique Haar probability measure, which is simply the $r$-fold product of the Haar measure on $\cU(\M_{n}(\C))$. For $i=1,\cdots,r$ we let $U_{i}^{(N)}\colon \cU(\M_{N}(\C))^{r}\to \cU(\M_{N}(\C))$ be the $i^{th}$ coordinate functions. Thinking of $U_{1}^{(N)},\cdots,U_{r}^{(N)}$ as $\cU(\M_{n}(\C))$-valued random variables, they are independent and each has distribution the Haar measure on $\cU(\M_{N}(\C))$. We will thus often refer to $U_{1}^{(N)},\cdots,U_{r}^{(N)}$ as independent Haar random unitaries. 

Suppose that $X$ is a Polish space and $d$ is a complete metric compatible with the topology of $d$, and $\mu$ a Borel probability measure on $X$. For $\varepsilon>0$, we define the \emph{concentration function} by 
\[\alpha_{X,d,\mu}(\varepsilon)=\sup\left\{\mu(X\setminus N_{\varepsilon}(A,d)):A\subseteq X \text{ is Borel and }\mu(A)\geq 1/2\right\}),\]
where $N_{\varepsilon}(A,d)=\bigcup_{x\in A}\{y\in X:d(x,y)<\varepsilon\}$. 
It will be helpful to recall the standard concentration of measure estimates for the unitary group (see \cite[Theorem 5.3]{Ledoux2001} and \cite[Theorem 17]{Meckes2013}).

\begin{thm} \label{thm: unitary concentration}
Fix $N\in \N$.
\begin{enumerate}[(i)]
\item for every Lipschitz function $f\colon \cU(\M_{N}(\C))^{r}\to \R$ we have that 
\[\P(|f(U)-\E[f(U)]|>\varepsilon)\leq 2e^{-\frac{N^{2}\varepsilon^{2}}{12 \|f\|_{\operatorname{Lip}}^{2}}},\]
where $\|f\|_{Lip}$ is the Lipschitz seminorm of $f$ with respect to the normalized Hilbert-Schmidt distance. 
\item We have that $\alpha_{\cU(\M_{N}(\C))^{r},\|\cdot\|_{2},\mu_{N}}(\varepsilon)\leq e^{-\varepsilon^{2}N^{2}/48}$, where $\mu_{N}$ is the Haar measure. 
\end{enumerate}
\end{thm}

One case of Voiculescu's asymptotic freeness theorem describe the large-$N$ behavior of $N \times N$ independent Haar random unitaries $U_1^{(N)}$, \dots, $U_r^{(N)}$.

First, let us recall the notion of convergence in non-commutative law.  Let $\mathbb{C} \mathbb{F}_r$ be the group algebra of $\mathbb{F}_r$. Given a tuple of unitaries $u_1$, \dots, $u_r$ in any tracial von Neumann algebra $(M,\tau)$ (including for instance $\mathbb{M}_N(\mathbb{C})$), one can evaluate any element $p \in \mathbb{C} [\mathbb{F}_r]$ on $u$ since there is a unique representation from $\mathbb{F}_r$ to $\mathcal{U}(M)$ sending the $j$th generator to $u_j$.  The \emph{non-commutative distribution} of $u = (u_1,\dots,u_r)$ is the linear functional $ell\: \mathbb{C} [\mathbb{F}_r] \to \C$ given by $\ell(p) = \tau(p(u_1,\dots,u_r))$, which is a tracial state.  The non-commutative distribution can also equivalently be viewed as a tracial state on the universal group $\mathrm{C}^*$-algebra $\mathrm{C}_u^*(\mathbb{F}_r)$ which contains $\mathbb{C} [\mathbb{F}_r]$ as a dense subalgebra.

The space of tracial states on can naturally be equipped with the weak-$*$ topology $\mathrm{C}_u^*(\mathbb{F}_r)$, and convergence of some sequence in non-commutative distribution refers to convergence of the corresponding tracial states in the weak-$*$ topology.  More explicitly, given a unitary tuples $u^{(k)} = (u_j^{(k)})_{j=1}^r$ for $k \in \N$ from a tracial von Neumann algebra $(M_k,\tau_k)$, and another $u = (u_j)_{j=1}^r$ from $(M,\tau)$, convergence in distribution means that $\tau_k(p(u^{(k)})) \to \tau(p(u))$ for every $p \in \mathbb{C} [\mathbb{F}_r]$.  Convergence in non-commutative distribution is closely related to embeddings into ultraproducts as well. Given an ultrafilter $\omega$ on $\mathbb{N}$, we have that $u^{(k)}$ converges in distribution to $u$ as $k \to \omega$ if and only if there exists a trace-preserving embedding $\iota$ from $(\mathrm{W}^*(u),\tau)$ into $\prod_{k \to \omega} (M_k,\tau_k)$ such that $\iota(u) = (u^{(k)})_{k \to \omega}$ (for details, see e.g. \cite[Lemma 5.10]{GJNS2022}).

Voiculescu's asymptotic freeness theorem \cite[Theorem 3.8]{VoicAsyFree} shows that Haar random unitaries converge in non-commutative distribution to the generators $u_1$, \dots, $u_r$ of the free group von Neumann algebra $L(\mathbb{F}_r)$ with the group trace $\tau_{\mathbb{F}_r}$.  We can thus restate this result in terms of ultraproduct embeddings as follows.

\begin{thm}
 Fix $r\in \N$, and let $U^{(N)}=(U^{(N)}_{i})_{i=1}^{r}$ be an $r$-tuple of i.i.d Haar distributed $N\times N$ unitary random matrices. Let $u=(u_{j})_{j=1}^{r}$ be an $r$-tuple of free Haar unitaries, and view $L(\F_{r})=W^{*}(u)$.  Then, for almost every $[U^{(N)}]_{N\in \N}\in \prod_{N}\cU(\M_{N}(\C)^{r})$ we have that for every free ultrafilter $\omega$, there is a (necessarily unique) embedding 
 \[\Theta_{U,\omega}\colon L(\F_{r})\to \prod_{N\to\omega}\M_{N}(\C)\]
 satisfying that $\Theta_{U}(u_{j})=(U_{j}^{(N)})_{N\to\omega}.$
\end{thm}

We caution that the map $U \mapsto \Theta_{U,\omega}$ is not measurable in any reasonable sense (see e.g.\ \cite[\S 6]{GJelementarydisintegration}), and thus one has to be slightly careful when defining sets based on the behavior of $\Theta_{U,\omega}$.  However, all the properties that we will be concerned with in this paper can be reduced to measurable properties of the sequence $(U^{(N)})_{N \in \mathbb{N}}$, as one can see from the proofs.

\section{A random Akemann-Ostrand property}\label{sec: Ayo AO}

 We recall the key $C^*$-algebraic property established by Akemann and Ostrand in \cite{AOProp} that the $*$-homomorphism 
\[
C^*_\lambda (\mathbb F_2) \otimes C^*_\rho (\mathbb F_2) \ni \sum_{i = 1}^n a_i \otimes x_i \mapsto \sum_{i = 1}^n a_i x_i \in \mathbb B(\ell^2 (\mathbb F_2)),
\] 
although not continuous with respect to the minimal tensor product, becomes continuous once it is composed with the quotient map onto the Calkin algebra $\mathbb B(\ell^2 (\mathbb F_2 )) / \mathbb K(\ell^2 (\mathbb F_2))$. A strengthened form of this property was established in the work of Ding-Peterson \cite[Theorem 7.20 (3)]{ding2023biexactvonneumannalgebras} (independently obtained also by Marrakchi in \cite[Corollary 3.7]{marrakchi2025kadison}): $$_{L(\mathbb{F}_n)}[L^2(L(\mathbb{F}_n)^\omega)\ominus L^2(L(\mathbb{F}_n))]_{L(\mathbb{F}_n)}\prec  _{L(\mathbb{F}_n)}[L^2(L(\mathbb{F}_n))\otimes L^2(L(\mathbb{F}_n))]_{L(\mathbb{F}_n)}.$$
The above result is the main inspiration behind the present article. In particular, the following Proposition is immediate:

\begin{prop}\label{prop:DingPeterson}
Suppose that the embedding $\Theta_{U,\omega}: L(\mathbb{F}_r) \to \prod_{N\to \omega}\mathbb{M}_N(\mathbb{C})$ is existential.
Then
\[
_{L(\mathbb{F}_r)}\left[L^2\left(\prod_{N\to \omega}\mathbb{M}_N(\mathbb{C})\right)\ominus L^2(L(\mathbb{F}_r)) \right]_{L(\mathbb{F}_r)}\prec _{L(\mathbb{F}_r)}[L^2(L(\mathbb{F}_r))\otimes L^2(L(\mathbb{F}_r))]_{L(\mathbb{F}_r)},
\]
where on the left-hand side we identify $L(\mathbb{F}_r)$ with $\Theta_{U,\omega}(L(\mathbb{F}_r))$.
\end{prop}

\begin{proof}
Assume the embedding is existential.  Note that to prove $_{L(\mathbb{F}_r)}\left[L^2\left(\prod_{N\to \omega}\mathbb{M}_N(\mathbb{C})\right)\ominus L^2(L(\mathbb{F}_r)) \right]_{L(\mathbb{F}_r)}$ is weakly coarse, it suffices to prove that for every separable $M_0$ with $\Theta_{U,\omega}(L(\mathbb{F}_r)) \leq M_0 \leq \prod_{N\to \omega}\mathbb{M}_N(\mathbb{C})$, the bimodule $_{L(\mathbb{F}_r)}\left[L^2(M_0)\ominus L^2(L(\mathbb{F}_r)) \right]_{L(\mathbb{F}_r)}$ is weakly coarse.  By assumption, there exists an embedding of $M_0$ into $L(\mathbb{F}_r)^\omega$ extending the diagonal embedding of $L(\mathbb{F}_r)$.  Thus,
\[
_{L(\mathbb{F}_r)}\left[L^2(M_0)\ominus L^2(L(\mathbb{F}_r)) \right]_{L(\mathbb{F}_r)} \prec _{L(\mathbb{F}_r)}\left[L^2(L(\mathbb{F}_r)^\omega) \ominus L^2(L(\mathbb{F}_r)) \right]_{L(\mathbb{F}_r)},
\]
which is weakly coarse by the result of Ding and Peterson \cite[Theorem 7.20 (3)]{ding2023biexactvonneumannalgebras}.
\end{proof}

The above Proposition motivates a \emph{random AO property} for free groups. The following conjecture is equivalent to saying that for every free ultrafilter, the random embedding gives a weakly orthocomplement bimodule (see Theorem \ref{thm: weak coarseness equivalent}). As we show later, this conjecture is sufficient to imply the Peterson--Thom property for $L(\F_{r})$. 

\begin{conj}\label{conj: the conj}
	\label{scmlskmc}
	Given $\alpha\in \C[\F_r]\otimes\C[\F_r]$, $u$ a $r$-tuple of free Haar unitaries, for almost every $[U^{(N)}]_N \in \prod_{N}\cU(\cM_{N}(\C)^{r})$ the following estimate holds:
	$$ \sup_{R>0} \inf_{F\subset \C[\F_r] \text{ finite}}\ \limsup_{N\to\infty} \sup_{\substack{A\in\M_N(\C), \norm{A}\leq R, \norm{A}_2\leq 1,\\ \forall P\in F, \tr(A^*P(U^N))=0}} \norm{\alpha(U^{(N)}\otimes 1_N, 1_N\otimes (U^{(N)})^{t}) \# A }_2 $$
    $$\leq \norm{\alpha(u\otimes 1, 1\otimes u^{-1})}_{C^*_{\lambda}(\F_r)\otimes_{\min} C^*_{\lambda}(\F_r) },$$
	where $(B\otimes C)\#A = BAC^T$.
\end{conj}

We will show that Conjecture \ref{conj: the conj} implies that the random embedding of $L(\F_{n})$ into ultraproduct of matrices is weakly coarse, for which we use the following well known proposition.

\begin{prop}\label{prop: folklore prop}
Let $M$ be a von Neumann algebra $A$ a unital weak$*$-dense $C^{*}$-subalgebra. Assume that $A$ is locally reflexive. Suppose that $\cH$ is a normal $M$-$M$ bimodule. If 
\[_A \cH_A\prec _A L^{2}(M)\otimes L^{2}(M)_A,\]
then
\[_M \cH_M\prec _M L^{2}(M)\otimes L^{2}(M)_M.\]
\end{prop}

\begin{proof}
Let $\pi\colon M\otimes_{\text{alg}} M^{\op}\to B(\cH)$ be the representation induced by the bimodule structure of $M$.
Fix $x_{1},\cdots,x_{k},y_{1},\cdots,y_{k}\in M$. Let $E$ be the operator system in $M$ generated by $x_{1},\cdots,x_{k},y_{1},\cdots,y_{k}$. Then by local reflexivity we may find a net
$\phi_{j}\colon E\to A$
of unital, completely positive maps so that $\phi_{j}(a)\to_{j}a$ SOT for all $a\in E$. 
By normality of $M$, and the uniform estimates $\|\phi_{j}(x_{l})\|\leq \|x_{l}\|$, $\|\phi_{j}(y_{l})\|\leq \|y_{l}\|$ we have
\[\sum_{l}\pi(\phi_{j}(x_{l})\otimes \phi_{j}^{\op}(y_{l}^{\op}))\to_{j}\sum_{l}\pi(x_{l}\otimes y_{l}^{\op}), \text{ in the strong operator topology.}\]
Thus,
\begin{align*}
\left\|\sum_{l}\pi(x_{l}\otimes y_{l}^{\op})\right\|&\leq \liminf_{j}\left\|\sum_{l}\pi(\phi_{j}(x_{l})\otimes \phi_{j}^{\op}(y_{l}^{\op}))\right\|\\
&\leq  \liminf_{j}\left\|\sum_{l}\phi_{j}(x_{l})\otimes \phi_{j}^{\op}(y_{l}^{\op})\right\|_{A\otimes_{\min}A^{\op}},    
\end{align*}
the last inequality following by our hypothesis that 
\[_A \cH_A\prec _A L^{2}(M)\otimes L^{2}(M)_A.\]
Since $\phi_{j}\otimes \phi_{j}^{\op}$ is u.c.p it follows that 
\[\left\|\sum_{l}\phi_{j}(x_{l})\otimes \phi_{j}^{\op}(y_{l}^{\op})\right\|_{A\otimes_{\min}A^{\op}}\leq \left\|\sum_{l}x_{l}\otimes y_{l}^{\op}\right\|_{M\otimes_{\min}M^{\op}},\]
since  $x_{1},\cdots,x_{k},y_{1},\cdots,y_{k}\in M$ were arbitrary, this proves that 
\[_M \cH_M\prec _M L^{2}(M)\otimes L^{2}(M)_M.\]

\end{proof}

We now explain how Conjecture is equivalent to weak coarseness of the random embedding of $L(\F_{r})$ into an ultraproduct of matrices. In particular, this proves Theorem \ref{thm:intro main thm} (\ref{item: intro item weak coarse equiv}).

\begin{thm}\label{thm: weak coarseness equivalent}
Conjecture \ref{conj: the conj} is equivalent to the following. For every $r\in \N$ and for almost every $U=(U^{(N)})_{n=1}^{\infty}\in \prod_{N}\cU(\M_{N}(\C)^{r})$ we have for every $\omega\in \beta\N\setminus \N$ that the embedding $\Theta_{U,\omega}$ is weakly coarse. 
\end{thm}

\begin{proof}
Set $(\cM,\tau)=\prod_{N\to\omega}(\M_{N}(\C),\tr).$
Note that Conjecture 3.2 implies that for almost every $U$ the following holds: for every $[A_{N}]_{N\in \N}\in \ell^{\infty}-\oplus_{N}\M_{N}(\C)$ with the property that for every polynomial $P\in \C(\F_{r})$, $\{N:\tr(A_{N}^{*}P(U^{(N)}))\ne 0\}$ is finite we have 
\[\limsup_{n\to\infty}\|\alpha(U^{(N)}\otimes 1_{N},1_{N}\otimes (U^{(N)})^{t})\#A_{N}\|_{2}\leq \|\alpha(u\otimes 1,1\otimes u)\|_{C^{*}_{\lambda}(\F_{r})\otimes_{\min}C^{*}_{\lambda}(\F_{r})}.\]
Note that if $a\in \cM$, then we may write $a=(A_{N})_{N\to\omega}$ where $[A_{N}]_{N\in \N}\in \ell^{\infty}-\oplus_{N}\M_{N}(\C)$ satisfies that for every polynomial $P\in \C(\F_{r})$ we have that $\{N:\tr(A_{N}^{*}P(U^{(N)}))\ne 0\}$ is finite. The above then shows that 
\[\|\alpha(\Theta_{U,\omega}(u)\otimes 1,1\otimes \Theta_{U,\omega}^{\op}(u^{\op}))\|\leq \|\alpha(u\otimes 1,1\otimes u)\|_{C^{*}_{\lambda}(\F_{r})\otimes C^{*}_{\lambda}(\F_{r})}.\]
By density of $\cM\ominus \Theta_{U,\omega}(L(\F_{r})$ in $L^{2}(\cM)\ominus L^{2}(\Theta_{U,\omega}(L(\F_{r})))$, and the fact that $C^{*}_{\lambda}(\F_{r})\cong C^{*}_{\lambda}(\F_{r})^{\op}$ via $\lambda(g)\mapsto \lambda(g^{-1})^{\op}$, we conclude that
\[_{C^{*}_{\lambda}(\F_{r})} L^{2}(\cM)\ominus L^{2}(\Theta_{U,\omega}(L(\F_{r})))_{C^{*}_{\lambda}(\F_{r})}\prec {C^{*}_{\lambda}(\F_{r})}_[\ell^{2}(\F_{r})\otimes \F_{r})]_{C^{*}_{\lambda}(\F_{r})}.\]
Note that $C^{*}_{\lambda}(\F_{r})$ is exact by \cite[Theorem 3.2]{DykemaExact}, and hence locally reflexive (see the comments at the end of the introduction of \cite{KirchbergExact} and \cite[Theorem 4.1]{EffrosHaagerupLR}).
Applying proposition \ref{prop: folklore prop},  we conclude that $\Theta_{U,\omega}$ is a weakly coarse embedding.

Conversely, suppose that for almost every $U$ we have that $\Theta_{U,\omega}$ is weakly coarse embedding for every free ultrafilter $\omega$. Suppose that Conjecture \ref{conj: the conj} is false. Since $\F_{r}$ is countable,  by a diagonal argument the negation of \ref{conj: the conj} combined with the assumption that the random embeddings is weakly coarse implies that  we may find a $U=(U^{(N)})_{N}\in \prod_{N}\cU(\M_{N}(\C))^{r}$ so that $\Theta_{U,\omega}$ is weakly coarse for every free ultrafilter and so that 
there exist an $R>0$, an $\alpha\in \C(\F_{r})\otimes \C(\F_{r})$ and a sequence $N_{1}<N_{2}<\cdots$ of natural numbers so that 
\[\lim_{k\to\infty}\sup_{\substack{\|A\|\leq R,\\ \tr(A^{*}P(U^{(N_{k})})=0 \forall P\in F},\|A\|_{2}\leq 1}\|\alpha(U^{(N_{k})}\otimes 1_{N_{k}},1_{N_{k}}\otimes (U^{(N_{k})})^{t})\# A\|_{2}\]
\[>\|\alpha(u\otimes 1,1\otimes u)\|_{C^{*}_{\lambda}(\F_{r})\otimes C^{*}_{\lambda}(\F_{r})},\]
for all finite $F\subseteq \F_{r}$. Choose a free ultrafilter $\omega$ on $\beta\N\setminus \N$ which extends $\lim_{k\to\infty}a_{N_{k}}$ for sequences $[a_{N}]_{N\in \N}\in \ell^{\infty}(\N)$ which have the property that $\lim_{k\to\infty}a_{N_{k}}$ exists.  
By another diagonal argument, we may find a sequence $(A_{N_{k}})_{k=1}^{\infty}\in \prod_{k}R\text{Ball}(\M_{N_{k}}(\C))$ with $\|A_{N_{k}}\|_{2}\leq 1$ for all $k$ and so that  for every finite $F\subseteq \F_{r}$ we have that 
\[\{k:\tr(A_{N_{k}}^{*}P(U^{(N_{k})}))\ne 0\}\]
is finite and for which 
\[\lim_{k\to\infty}\|\alpha(U^{(N_{k})}\otimes 1_{N_{k}},1_{N_{k}}\otimes (U^{(N_{k})})^{t})\# A_{N_{k}}\|_{2}>\|\alpha(u\otimes 1,1\otimes u)\|_{C^{*}_{\lambda}(\F_{r})\otimes C^{*}_{\lambda}(\F_{r})}.\]
Set $A_{N}=0$ if $N\ne N_{k}$ for any $k$. Then $a=(A_{N})_{N\to\omega}\in \cM\ominus \Theta_{U,\omega}(L(\F_{r}))$ with $\|a\|_{2}\leq 1$ and 
\[\|\alpha(\Theta_{U,\omega}(u)\otimes 1,1\otimes \Theta_{U,\omega}^{\op}(u^{\op}))\# a\|_{2}>\|\alpha(u\otimes 1,1\otimes u)\|_{C^{*}_{\lambda}(\F_{r})\otimes C^{*}_{\lambda}(\F_{r})}.\]
This contradicts our assumption that $\Theta_{U,\omega}$ is weakly coarse for every free ultrafilter $\omega$. 
    
\end{proof}

\section{From Conjecture \ref{intro conj: the conj} to Peterson-Thom}

Recall the definition of the concentration function $\alpha_{(X,d,\mu)}$ from the preliminaries. Suppose we are given sequences:
\begin{itemize}
    \item $(X_{n})_{n=1}^{\infty}$ of Polish spaces, together with
    \item  compatible continuous metrics $(d_{n})_{n=1}^{\infty}$,
    \item and Borel probability measures $(\mu_{n})_{n=1}^{\infty}$,
    \item and a sequence of positive real numbers $(s_{n})_{n=1}^{\infty}$.
\end{itemize}  Then we say that $(X_{n},\mu_{n},d_{n})$ satisfies \emph{exponential concentration of measure at scale $s_{n}$} if 
\[\limsup_{n\to\infty}\frac{-1}{s_{n}}\log \alpha_{(X_{n},d_{n},\mu_{n})}(\varepsilon)>0\]
for every $\varepsilon>0$.  


The following is \cite[Theorem 4.3]{HayesPT}.

\begin{thm}
Suppose we are given a tracial von Neumann algebra $(M,\tau)$, a countable index set $J,$ and an $x\in M^{J}$ with $W^{*}(x)=M.$ Suppose $R\in [0,\infty)^{J}$ satisfies $\|x_{j}\|_{\infty}\leq R_{j}$ for all $j\in J.$ Assume we are given a sequence of natural numbers $n(k)\to\infty,$ and a sequence  $\mu^{(k)}\in \Prob(\M_{N(k)}(\C)^{J})$ such that
\begin{itemize}
    \item $\sum_{k}\mu^{(k)}\left(\left(\prod_{j\in J}\{A\in \M_{N(k)}(\C)^{J}:\|A_{j}\|_{\infty}\leq R_{j}\right)^{c}\right)<\infty$.
    \item $\mu^{(k)}(\Gamma^{(N(k))}_{R}(\mathcal{O}))\to 1$ for all weak$^{*}$ neighborhoods $\mathcal{O}$ of $\ell_{x}$ in $\Sigma_{R,J}$,
    \item $\mu^{(k)}$ has exponential concentration with scale $N(k)^{2}.$
    \end{itemize}
    Then:
\begin{enumerate}[(i)]
    \item there is a conull subset $\Omega_{0}\subseteq \prod_{k}M_{N(k)}(\C)^{J}$ so that for any $A=(A^{(k)})_{k}\in \Omega_{0},$ and for every free ultrafilter $\omega$ on $\N,$ there is a unique trace-preserving $*$-homomorphism $\Theta_{A,\omega}\colon M\to\prod_{k\to\omega}(M_{N(k)}(\C),\tr)$ so that
    \[\Theta_{A,\omega}(P(x))=(P(A^{(k)}))_{k\to\omega}\,\]for all every noncommutative $*$-polynomial $P$ in abstract variables $(T_{j})_{j\in J}$.\label{item:as defined homom}
    \item If $Q\leq M$ satisfies $h(Q:M)\leq 0,$ then there is a conull subset $\Omega\subseteq \Omega_{0}$ so that for all $A,B\in \Omega$ and for every free ultrafilter $\omega$ on $\N,$ we have that  $\Theta_{A,\omega}\big|_{Q},$ $\Theta_{B,\omega}\big|_{Q}$ are unitarily conjugate.  \label{item:almost sure conjugacy}
    \end{enumerate}
    \end{thm}
The above result implies the following result which can be viewed as a \emph{random} version of Jung's theorem \cite{JungTubularity}.

\begin{cor}\label{cor: random jung collapse}
Fix $r\in \N$. Let $a_{1},\cdots,a_{r}$ be free generators of $L(\F_{r})$. Then:
\begin{enumerate}[(i)]
\item There is a conull subset $\Omega_{0}\subseteq \prod_{N}\cU(\M_{N}(\C))^{r}$ (with respect to Haar measure) so that for all $U\in \Omega$ and all $\omega\in \beta \N\setminus \N$ we there is a unique trace-preserving $*$-homomorphism $\Theta_{U,\omega}\colon L(\F_{r})\to \prod_{N\to\omega}\M_{N}(\C)$ satisfying $\Theta_{U,\omega}(\lambda(a_{j}))=(U_{n,j})_{n\to\omega}$ for all $j=1,\cdots,r$.
\item If $Q\leq L(\F_{r})$ has $h(Q:M)=0$, then there is a conull $\Omega\subseteq \Omega_{0}$ so that for all $U,V\in \Omega$ and for all $\omega\in \beta \N\setminus\N$, we have that $\Theta_{U,\omega}|_{Q}$ and $\Theta_{V,\omega}|_{Q}$ are unitarily conjugate.
\end{enumerate}
\end{cor}

We now show that if Conjecture \ref{conj: the conj} is true, then the above corollaries combine to prove that every nonamenable subalgebra of a free group factor has positive $1$-bounded entropy. As is well known, this implies the Peterson-Thom conjecture (see \cite[Proposition 2.7]{HayesPT}). In particular, the following proves Theorem \ref{thm:intro main thm} (\ref{item: intro item PT}).

\begin{cor}\label{cor: PT corollary}
Adopt notation as in Corollary \ref{cor: random jung collapse}  Fix $\omega\in \beta \N\setminus \N$,  and set $(\cM,\tau)=\prod_{N\to\omega}(\M_{N}(\C),\tr)$.
Suppose that for every $r\in \N$, there is a conull subset $\Xi_{r}\subseteq \prod_{n}\cU(\M_{N}(\C))^{r}$ so that for every $U\in \Xi_{r}$ the $L(\F_{r})-L(\F_{r})$ bimodule $L^{2}(\cM)\ominus L^{2}(\Theta_{U,\omega}(L(\F_{r})))$ is weakly coarse. Then for every $r\in \N$ and every nonamenable $Q\leq L(\F_{r})$ we have that  $h(Q:L(\F_{r}))>0$.
\end{cor}

\begin{proof}

Let $\Omega$ be as in Corollary \ref{cor: random jung collapse}. We view $\prod_{n}\cU(\M_{N}(\C))^{2r}=(\prod_{n}\cU(\M_{N}(\C))^{r})^{2},$
and set $\Upsilon=\Xi_{2r}\cap (\Omega\times \Omega)$. Then $\Upsilon$ is conull. Fix $(U,V)\in \Upsilon$. View $L^{2}(\cM)$ as an $L(\F_{r})-L(\F_{r})$ bimodule via $x\cdot \xi \cdot y=\Theta_{U,\omega}(x)\xi\Theta_{V,\omega}(y)$. Suppose that $W\in \cU(\cM)$ and $\Ad(W^{*})\circ \Theta_{U,\omega}|_{Q}=\Theta_{V,\omega}|_{Q}$. Then $W$ is a central and tracial vector for the $Q$-$Q$ bimodule structure. We may split $L^{2}(\cM)$ as a $Q$-$Q$ bimodule via
\[L^{2}(\cM)=L^{2}(\Theta_{(U,V),\omega}(L(\F_{2r})))\oplus [L^{2}(\cM)\ominus L^{2}(\Theta_{(U,V),\omega}(L(\F_{2r})))].\]
As a $Q$-$Q$ bimodule we have that $L^{2}(\Theta_{(U,V),\omega}(L(\F_{2r})))\cong \ell^{2}(\F_{2r})$, where $x\in Q$ acts on the left by left multiplication by $x*1$ and on the right by $1*x$.
Note that $\ell^{2}(\F_{2r})$ is as an $L(\F_{r})$-$L(\F_{r})$ bimodule is an infinite coarse. Indeed, the action of $\F_{r}\times \F_{r}$ on $\F_{2r}=\ip{a_{1},\cdots,a_{2r}}$ given by $(w_{1},\omega_{2})\cdots x=w_{1}(a_{1},\cdots,a_{r})xw_{2}(a_{r+1},\cdots,a_{2r})^{-1}$ has trivial stabilizer (one can more generally show that $L^{2}(M*M)$ is a coarse $M$-$M$ bimodule for any tracial von Neumann algebra $(M,\tau)$, but we will not need this here). 
Restricting the $L(\F_{r})$-$L(\F_{r})$ bimodule structure to $Q$-$Q$, we see that the $Q$-$Q$ bimodule  $\ell^{2}(\F_{2r})$ is coarse. 
Restricting the $L(\F_{2r})-L(\F_{2r})$ bimodule structure to $Q*1-1*Q$ we have that $L^{2}(\cM)\ominus L^{2}(\Theta_{(U,V),\omega}(L(\F_{2r})))$ is weakly coarse as a $Q$-$Q$ bimodule. Thus $L^{2}(\cM)$ is weakly coarse as a $Q$-$Q$ bimodule. Since $W$ is a central and tracial vector, we have that the trivial $Q$-$Q$ bimodule  $L^{2}(Q)$ is contained in the $Q$-$Q$ bimodule $L^{2}(\cM)$. So the trivial bimodule for $Q$ is weakly contained in the coarse bimodule, and thus $Q$ is amenable by \cite[Remark 5.29]{Connes}. 

\end{proof}

\section{Existential embedding conjecture}

In this section, we explain how Conjecture \ref{intro conj: random optimization conjecture} is a statement about existential embeddings. We recall the statement of the conjecture here.

\begin{conj} \label{conj: random optimization conjecture}
Let $r \in \mathbb{N}$.  Let $U_1^{(N)}$, \dots, $U_r^{(N)}$ be independent Haar random unitaries.  Let $u_1$, \dots, $u_r$ denote the canonical generators of $L(\mathbb{F}_r)$.  Then for every $s \in \N$ and every self-adjoint $*$-polynomial in $r + s$ variables, we have
\begin{multline} \label{eq: optimization conjecture}
\lim_{N \to \infty} \mathbb{E} \left[ \sup_{V_1,\dots,V_s \in \mathcal{U}(M_N(\mathbb{C}))} \tr_n(p(U_1^{(N)},\dots,U_r^{(N)},V_1,\dots,V_s)) \right] \\ = \sup_{v_1,\dots,v_s \in \mathcal{U}(L(\mathbb{F}_r))} \tau(p(u_1,\dots,u_r,v_1,\dots,v_s)).
\end{multline}
Here note that the candidates $V_j$ in the sup are allowed to depend on $U^{(N)}$.
\end{conj}

\begin{remark}
The following inequality is automatically true:
\begin{multline} \label{eq: optimization conjecture easy direction}
\liminf_{N \to \infty} \mathbb{E} \left[ \sup_{V_1,\dots,V_s \in \mathcal{U}(M_N(\mathbb{C}))} \tr_n(p(U_1^{(N)},\dots,U_r^{(N)},V_1,\dots,V_s)) \right] \\ \geq \sup_{v_1,\dots,v_s \in \mathcal{U}(L(\mathbb{F}_r))} \tau(p(u_1,\dots,u_r,v_1,\dots,v_s)).
\end{multline}
Indeed, $v_1$, \dots, $v_s \in L(\mathbb{F}_r)$, then they can be approximated in $2$-norm by $*$-polynomials $w_j = p_j(u_1,\dots,u_r)$ for $j = 1$, \dots, $s$.  Letting $W_j^{(N)} = p_j(U_1^{(N)},\dots,U_r^{(N)})$, we have
\begin{multline*}
\liminf_{N \to \infty} \mathbb{E} \left[ \sup_{V_1,\dots,V_s \in \mathcal{U}(M_N(\mathbb{C}))} \tr_n(p(U_1^{(N)},\dots,U_r^{(N)},V_1,\dots,V_s)) \right] \\ \geq \lim_{n \to \infty} \mathbb{E} \tr_n(p(U_1^{(N)},\dots,U_r^{(N)},W_1^{(N)},\dots,W_s^{(N)})) \\
= \tau(p(u_1,\dots,u_r,w_1,\dots,w_s))
\end{multline*}
by Voiculescu's asymptotic freeness theorem \cite{VoicAsyFree}.  Since $w_j$ can be taken arbitrarily close to $v_j$, we have \eqref{eq: optimization conjecture easy direction}. Hence, the conjecture is equivalent to
\begin{multline} \label{eq: optimization conjecture hard direction}
\limsup_{N \to \infty} \mathbb{E} \left[ \sup_{V_1,\dots,V_s \in \mathcal{U}(M_N(\mathbb{C}))} \tr_n(p(U_1^{(N)},\dots,U_r^{(N)},V_1,\dots,V_s)) \right] \\ \leq \sup_{v_1,\dots,v_s \in \mathcal{U}(L(\mathbb{F}_r))} \tau(p(u_1,\dots,u_r,v_1,\dots,v_s)).
\end{multline}
Moreover, since the above function is Lipschitz in $u_1$, \dots, $u_r$ in the unitary group, concentration of the Haar measure on $\mathcal{U}(M_N(\mathbb{C}))$ implies that the statement holds for the expectation if and only if it holds almost surely.
\end{remark}

\begin{prop} \label{prop: existential versus optimization}
Let $U_1^{(N)}$, \dots, $U_r^{(N)}$ be independent Haar random unitaries.  Recall that almost surely (i.e. on a co-null subset of $\prod_{N \in \mathbb{N}} \mathcal{U}(M_N(\mathbb{C}))$), for every free ultrafilter $\omega \in \beta \mathbb{N} \setminus \mathbb{N}$, there is an embedding $\Theta_{U,\omega}: L(\mathbb{F}_r) \to \prod_{N \to \omega} M_N(\mathbb{C})$.

Conjecture \ref{conj: random optimization conjecture} holds if and only if, for every ultrafilter $\omega$, for almost every $U \in \prod_{N \in \mathbb{N}} \mathcal{U}(M_N(\mathbb{C}))$, $\Theta_{U,\omega}$ is an existential embedding of tracial von Neumann algebras.
\end{prop}

The first ingredient is a convex analysis argument, which is what ultimately allows us to look at only individual traces in Conjecture \ref{conj: random optimization conjecture}, rather than more complicated functions depending on several traces.

\begin{lem} \label{lem: extreme points}
Let $A$ be a $\mathrm{C}^*$-algebra, let $A_0$ be a dense $*$-subalgebra, and let $\tau$ be an extreme point in $\mathcal{T}(A)$.  For every neighborhood $\mathcal{O}$ of $\tau$, there exists $a \in A_0$ self-adjoint and $\varepsilon > 0$ such that
\[
\forall \varphi \in \mathcal{T}(A), \varphi(a) > \tau(a) - \varepsilon \implies \varphi \in \mathcal{O}.
\]
\end{lem}

\begin{proof}
Let $\mathcal{K}_0 = \mathcal{T}(A) \setminus \mathcal{O}$.  Let $\mathcal{P}(\mathcal{K}_0)$ be the space of Borel probability measures on $\mathcal{K}_0$ with the weak-$*$ topology.  Let $F: \mathcal{P}(\mathcal{K}_0) \to \mathcal{T}(A)$ be the map that sends a measure $\mu$ to its barycenter, that is,
\[
F(\mu)(a) = \int_{\mathcal{K}_0} \varphi(a)\,d\mu(\varphi).
\]
Let $\mathcal{K}$ be the image of $F$.  Since $F$ is continuous and $\mathcal{P}(\mathcal{K}_0)$ is compact, $\mathcal{K}$ is compact.  And $\mathcal{K}$ is also convex since the map $F$ is affine.  Clearly, $\mathcal{K}$ contains $\mathcal{K}_0$, since $F(\delta_\varphi) = \varphi$.

Note that $\tau$ is not in $\mathcal{K}$.  Indeed, since $\tau$ is extreme, $\delta_\tau$ is the only measure whose barycenter is $\tau$.  So $\tau$ is a point outside the compact convex set $\mathcal{K}$.  By the Hahn--Banach theorem, there exists some $a \in A$ and $\varepsilon > 0$ such that $\re \varphi(a) \leq \re \tau(a) - \varepsilon$ for all $\varphi \in \mathcal{K}$.  By replacing $a$ with $(a+a^*)/2$ assume $a$ is self-adjoint.  Clearly, we can also choose $a$ from the dense subset $A_0$ at the expense of replacing $\varepsilon$ with $2 \varepsilon$.  Thus, $\varphi(a) > \tau(a) - \varepsilon$ implies that $\varphi$ is not in $\mathcal{K}$, and hence $\varphi \in \mathcal{O}$.
\end{proof}

\begin{proof}[Proof of Proposition \ref{prop: existential versus optimization}]
Note that every $*$-polynomial is Lipschitz with respect to $\norm{\cdot}_2$ on $\mathcal{U}(M_N(\mathbb{C}))^r$.  Moreover, since the class of $L$-Lipschitz functions is closed under suprema, the map
\[
\varphi(U_1,\dots,U_r) = \sup_{V_1, \dots, V_s} \tr_N(p(U_1,\dots,U_r,V_1,\dots,V_s)
\]
is also Lipschitz on $\mathcal{U}(M_N(\mathbb{C}))^r$ with a Lipschitz constant independent of $N$.  Therefore, by concentration of measure (Theorem \ref{thm: unitary concentration} (i)), we see that almost surely
\begin{equation} \label{eq: concentration for existential conjecture}
\lim_{N \to \infty} |\varphi(U_1^{(N)},\dots,U_r^{(N)}) - \mathbb{E} \varphi(U_1^{(N)},\dots,U_r^{(N)})| = 0.
\end{equation}
Moreover, by considering a countable set of polynomials that are dense in $\mathrm{C}_u^*(\mathbb{F}_{r+s})$ for each $s$, we see that almost surely this condition holds for all $p$ and $s$ simultaneously.

Assume \eqref{eq: optimization conjecture} holds. 
Fix an ultrafilter $\omega$ and assume that $U = (U^{(N)})_N \in \prod_{N \in \mathbb{N}} \mathcal{U}(M_N(\mathbb{C}))$ is in the almost sure event above.  Consider a separable von Neumann algebra $M$ with
\[
\Theta_{U,\omega}(L(\mathbb{F}_r)) = \mathrm{W}^*(u_1,\dots,u_r) \subseteq M \subseteq \prod_{N \to \omega} \mathbb{M}_N,
\]
and we will show that there exists an embedding $\iota: M \to L(\mathbb{F}_r)^\omega$ such that $\iota \circ \Theta_{U,\omega} = \Delta$.  By diagonalization, it suffices to consider the case when $M$ is finitely generated, say $M$ is generated by $\Theta_{U,\omega}(u_1)$, \dots, $\Theta_{U,\omega}(u_r)$ and unitaries $w_1$, \dots, $w_s$.

Let $A = \mathrm{C}_u^*(\mathbb{F}_{r+s})$ and $A_0 = \mathbb{C} \mathbb{F}_{r+s}$, and let $\sigma_0 \in \mathcal{T}(A)$ be the non-commutative distribution of $(u_1,\dots,u_r,w_1,\dots,w_s)$ (see \S \ref{sec:preliminaries}).  Let $(\mathcal{O}_k)_{k \in \N}$ be a sequence of neighborhoods of $\sigma_0$ such that $\mathcal{O}_{k+1} \subseteq \mathcal{O}_k$ and $\bigcap_{k \in \N} \mathcal{O}_k = \varnothing$.  Since $\mathrm{W}^*(u_1,\dots,u_r)$ is an irreducible subfactor of the matrix ultraproduct, we know that $M$ is a factor.  Therefore, $\sigma_0$ is extreme in $\mathcal{T}(A)$.  Hence, by Lemma \ref{lem: extreme points}, there exists $p_k \in \C \mathbb{F}_{r+s}$ such that $\sigma(p_k) > \sigma_0(p_k) - \varepsilon$ implies that $\sigma \in \mathcal{O}_k$.

Write $w_j = (W_j^{(N)})_{N \to \omega}$ where $W_j^{(N)} \in \mathcal{U}(M_N(\mathbb{C}))$.  In the almost sure event defined above,
\begin{align*}
\sup_{v_1, \dots, v_s \in \mathcal{U}(L(\mathbb{F}_m))} &\tau_{L(\mathbb{F}_r)}(p_k(u_1,\dots,u_r,v_1,\dots,v_s)) \\
&\geq \limsup_{N \to \infty} \sup_{V_1,\dots,V_s \in \mathcal{U}(M_N(\mathbb{C}))} \tr_N(p_k(U_1^{(N)},\dots,U_r^{(N)},V_1,\dots,V_s)) \\
&\geq \lim_{N \to \omega} \tr_N(p_k(U_1^{(N)},\dots,U_r^{(N)},W_1^{(N)},\dots,W_s^{(N)})) \\
&= \tau_M(p_k(u_1,\dots,u_r,w_1,\dots,w_s)).
\end{align*}
Hence, there exist $v_1^{(k)}$, \dots, $v_s^{(k)}$ in $L(\mathbb{F}_r)$ such that
\[
\tau_{L(\mathbb{F}_r)}(p_k(u_1,\dots,u_r,v_1^{(k)},\dots,v_s^{(k)})) > \varphi_0(p_k) - \varepsilon.
\]
Thus, letting $\sigma_k \in \mathcal{T}(A)$ be the non-commutative law of $(u_1,\dots,u_r,v_1^{(k)},\dots,v_s^{(k)})$, we have $\sigma_k \in \mathcal{O}_k$.  Therefore, $\lim_{k \to \omega} \sigma_k = \sigma_0$.  Let $\tilde{w}_j = (v_j^{(k)})_{k \to \omega}$ in the ultrapower $L(\mathbb{F}_r)^\omega$.  Then the law of $(u_1,\dots,u_r,\tilde{w}_1,\dots,\tilde{w}_s)$ is $\sigma_0$, and therefore, there is an embedding $\iota$ of $M$ into $L(\mathbb{F}_r)^\omega$ given by
\[
\iota(u_j) = u_j, \qquad \iota(w_j) = \tilde{w}_j.
\]
This proves the embedding is existential as desired.

Conversely, assume that the embedding is existential for almost every $U \in \prod_{n \in \N} \mathcal{U}(M_N(\mathbb{C}))^r$.  Fix $U$ in the almost sure event given above, and fix a self-adjoint non-commutative $*$-polynomial $p$ in $r + s$ variables.  Let $V_1^{(N)}$, \dots, $V_s^{(N)} \in \mathcal{U}(M_N(\mathbb{C}))$ such that
\begin{multline} \label{eq: number 1}
\tr_N[p(U_1^{(N)},\dots,U_r^{(N)}(\omega),W_1^{(N)},\dots,W_s^{(N)})] \\
= \sup_{V_1,\dots,V_s \in \mathcal{U}(M_N(\mathbb{C}))} \tr_n[p(U_1^{(N)}(\omega),\dots,U_r^{(N)}(\omega),V_1,\dots,V_s)].
\end{multline}
Let
\[
w_j = (W_j^{(N)})_{N \to \omega} \in \prod_{N \to \omega} M_N(\mathbb{C}).
\]
Thus,
\begin{multline} \label{eq: number 2}
\tau[p(\Theta_{U,\omega}(u_1),\dots,\Theta_{U,\omega}(u_r),w_1,\dots,w_s)] \\
=\lim_{N \to \omega} \tr_N[p(U_1^{(N)},\dots,U_r^{(N)},W_1^{(N)},\dots,W_s^{(N)})].
\end{multline}
Let $M = \mathrm{W}^*(\Theta_{U,\omega}(u_1),\dots,\Theta_{U,\omega}(u_r),v_1,\dots,v_s) \supseteq \Theta_{U,\omega}(L(\mathbb{F}_r))$.  Since the embedding $\Theta_{U,\omega}$ is existential, there is an embedding $\iota: M \to L(\mathbb{F}_r)^\omega$ such that $\iota \circ \Theta_{U,\omega} = \Delta$.  
By a diagonal argument (or because the diagonal embedding $L(\mathbb{F}_r) \to L(\mathbb{F}_r)^\omega$ is existential), we have that 
\begin{align*}
\sup_{v_1, \dots, v_s \in \mathcal{U}(L(\mathbb{F}_r))}& \tau(p(u_1,\dots,u_r,v_1,\dots,v_s))\\
&=\sup_{v_1,\dots,v_s \in \mathcal{U}(L(\mathbb{F}_r)^{\cU})} \tau(p(\Delta(u_1),\dots,\Delta(u_r),v_1,\dots,v_s)),
\end{align*}
 with $\Delta$ the diagonal embedding. Thus:
\begin{align}
\sup_{v_1, \dots, v_s \in \mathcal{U}(L(\mathbb{F}_r))} & \tau(p(u_1,\dots,u_r,v_1,\dots,v_s)) \label{eq: number 3} \\
&=
\sup_{v_1,\dots,v_s \in \mathcal{U}(L(\mathbb{F}_r)^{\cU})} \tau(p(\Delta(u_1),\dots,\Delta(u_r),v_1,\dots,v_s)) \nonumber \\
&\geq \tau(p(\Delta(u_1),\dots,\Delta(u_r),\iota(w_1),\dots,\iota(w_s))) \nonumber \\
&= \tau(p(\Theta_\omega(u_1),\dots,\Theta_\omega(u_r),w_1,\dots,w_r)). \nonumber
\end{align}
By \eqref{eq: number 3}, \eqref{eq: number 2}, \eqref{eq: number 1}, and \eqref{eq: concentration for existential conjecture},
\begin{align*}
\sup_{v_1, \dots, v_s \in \mathcal{U}(L(\mathbb{F}_r))} &\tau(p(u_1,\dots,u_r,v_1,\dots,v_s)) \\
&\geq 
\lim_{n \to \omega} \sup_{V_1,\dots,V_s \in \mathcal{U}(M_N(\mathbb{C}))} \tr_N[p(U_1^{(N)},\dots,U_r^{(N)},V_1,\dots,V_s)] \\
&= \lim_{N \to \omega} \mathbb{E} \left[ \sup_{V_1,\dots,V_s \in \mathcal{U}(M_N(\mathbb{C}))} \tr_N[p(U_1^{(N)},\dots,U_r^{(N)},V_1,\dots,V_s)] \right].
\end{align*}
Since the ultrafilter $\omega$ was arbitrary, \eqref{eq: optimization conjecture hard direction} holds.
\end{proof}

We close the paper by connecting the two Conjectures, and in particular finish the proof of Theorem \ref{thm:intro main thm} (\ref{item: intro conjecture bridge}).

\begin{thm}
Conjecture \ref{conj: random optimization conjecture} implies Conjecture \ref{conj: the conj}.   
\end{thm}

\begin{proof}
   If conjecture \ref{conj: random optimization conjecture} holds, then by Proposition \ref{prop: existential versus optimization}, we have that almost surely $\Theta_{U,\omega}$ is existential for all $\omega$. By Proposition \ref{prop:DingPeterson}, this implies that $\Theta_{U,\omega}$ is almost surely weakly coarse for all $\omega.$ Hence, by Theorem \ref{thm: weak coarseness equivalent} Conjecture \ref{conj: the conj} holds. 
   \end{proof}

\providecommand{\bysame}{\leavevmode\hbox to3em{\hrulefill}\thinspace}
\providecommand{\MR}{\relax\ifhmode\unskip\space\fi MR }
\providecommand{\MRhref}[2]{%
  \href{http://www.ams.org/mathscinet-getitem?mr=#1}{#2}
}
\providecommand{\href}[2]{#2}

%

\end{document}